\documentclass{article}
\usepackage[utf8]{inputenc}
\usepackage{amsmath,amssymb,amsfonts,amsthm,comment}
\newtheorem{theorem}{Theorem}
\newtheorem{proposition}{Proposition}
\newtheorem{lemma}{Lemma}
\usepackage{graphicx}
\newcommand{\diag}{\hbox{\rm \,diag}}
\newcommand*\xbar[1]{%
  \hbox{%
    \vbox{%
      \hrule height 0.5pt 
      \kern0.5ex
      \hbox{%
        \kern-0.1em
        \ensuremath{#1}%
        \kern-0.1em
      }%
    }%
  }%
} 
\newcommand{\one}{{\bf 1}}

\title{A defect-correction algorithm for quadratic matrix equations, with applications to quasi-Toeplitz matrices\footnote{Work supported by GNCS of INdAM.}}
\author{Dario A. Bini\footnote{Dipartimento di Matematica, Universit\`a di Pisa ({email: \tt dario.bini@unipi.it})}, and Beatrice Meini\footnote{Dipartimento di Matematica, Universit\`a di Pisa ({email: \tt beatrice.meini@unipi.it})}}

\begin{document}
\maketitle
\begin{abstract}
A defect correction formula for quadratic matrix equations of the kind $A_1X^2+A_0X+A_{-1}=0$ is presented. This formula, expressed by means of an invariant subspace of a suitable pencil, allows us to introduce a modification of the Structure-preserving Doubling Algorithm (SDA), that enables refining an initial approximation 
to the sought solution. This modification provides substantial advantages, in terms of convergence acceleration, in the solution of equations coming from stochastic models, by choosing a stochastic matrix as the initial approximation. An application to solving random walks in the quarter plane is shown, where the coefficients $A_{-1},A_0,A_1$ are quasi-Toeplitz matrices of infinite size.
\end{abstract}

\section{Introduction}
Let $n\in\mathbb N\cup\{\infty\}$ and let $A_k$, $k=-1,0,1$, be $n\times n$  matrices. The quadratic matrix equation 
\begin{equation}\label{eq:mateq}
A_1X^2+A_0X+A_{-1}=0
\end{equation}
is related to the quadratic eigenvalue problem \cite{guo}, \cite{HK} and is encountered in many applications from different areas of applied mathematics \cite{lr:book}, \cite{tiss};  in the last decade, algorithms for its solution have received much interest in the literature. In particular, in the analysis of quasi birth-death (QBD) processes \cite{lr:book} and more specifically in  bidimensional random walks, the matrix coefficients $A_{-1},A_0,A_1$ have the form
$A_{-1}=-B_{-1}$, $A_0=I-B_0$, $A_1=-B_1$, where $B_i$ are tridiagonal matrices having nonnegative entries and $B=B_{-1}+B_0+B_1$ is row stochastic, that is, $B\one=\one$, where $\one$ is the vector with all ones.
In this framework, the solution of interest is the minimal nonnegative solution $X=G$ that always exists \cite{neuts}. In certain stochastic processes characterized by an infinite number of states \cite{ozawa}, the size of the matrix coefficients $A_i$ as well as of the solution $G$ is infinite \cite{bmmr:sisc}. In a more general framework, it is assumed that there exists a solution $G$ having minimal spectral radius and the goal is its numerical approximation.

Several algorithms for computing $G$ have been introduced in the literature. Besides the fixed point iterations having a linear convergence, see \cite[Chapter 6]{blm:book}, quadratically convergent algorithms like Newton's iteration \cite{latouche}, cyclic reduction (CR) and logarithmic reduction (LR) \cite[Chapter 7]{blm:book} or Structure-preserving Doubling Algorithms (SDA) \cite[Chapter 5]{bim:book}, \cite{li:book}, have been considered and analysed in the literature.

Fixed point iterations, defined by the recurrence $X_{k+1}=F(X_k)$ for a given matrix function $F(X)$ and for a given initial approximation $X_0$,  have a relatively low cost per step, a typically linear convergence, and provide the possibility to refine a given approximation $X_0$.
Newton's iteration still allows the choice of a starting approximation to be refined, has the advantage to provide a quadratic convergence but its computational cost per step is generally much higher than the cost of the standard fixed-point iterations, since a Sylvester equation must be solved at each step.
On the other hand, CR, LR,  and SDA, have quadratic convergence, a relatively low cost per step, so that they are much more convenient than Newton iteration. But unfortunately, unlike fixed-point iterations,  they do not allow to choose an assigned initial approximation $X_0$. This is an annoying drawback of this class of methods.

In this paper, by following the ideas of \cite{defectBG} and \cite{defectM}, starting from an approximation $\widetilde G$ of the sought solution $G$, we derive an equation for the defect $H=G-\widetilde G$ and express $H$ in terms of the invariant subspace of a suitable pencil. By relying on this formulation,  we provide a modification of SDA that allows us to refine an initial approximation $\widetilde G$. We prove some convergence results and, in the case of problems stemming from stochastic processes, we show that, under a suitable choice of $\widetilde G$, the convergence speed of SDA can be further improved. Finally, we show an application to the analysis of random walks in the quarter plane, where the matrix coefficients $A_i$, $i=-1,0,1$, as well as the sought solution $G$, are infinite matrices endowed of the quasi-Toeplitz structure (QT matrices). In this framework, there are situations where CR, LR and SDA fail to converge if applied in the customary way, whereas, under a suitable choice of the starting approximation $\widetilde G$,  our modified version of SDA  converges in a few iteration steps.

By means of numerical experiments we show that the test problems discussed in \cite{bmmr:sisc}, concerning infinite dimensional problems,  can be efficiently solved by our modified version of SDA with a CPU time that is much inferior to the time needed by the previously available algorithms \cite{bmm20}.

The paper is organized as follows. In Section \ref{sec:prel}, we summarise some basic properties of SDA and of invariant subspaces of pencils, moreover  we recall the SDA iteration together with its convergence properties. In Section \ref{sec:lin}, we first reformulate the quadratic matrix equation in terms  of invariant subspaces of a linear pencil, then we introduce the defect equation and express it in terms of an invariant subspace, finally we introduce and analyse the modification of SDA to refine a given approximation of the solution.
In Section~\ref{sec:stoc}, we consider the case where the equation originates from a stochastic model and show that choosing $\widetilde G$ stochastic yields a substantial acceleration of the convergence.
In the same section, we introduce a further modification that, besides providing a further acceleration of the convergence speed, allows us to control the convergence by means of a reliable and cheap condition to halt the iterations.
Section \ref{sec:qt} shows an application of our technique to solving the quadratic matrix equation in the infinite dimensional case where coefficients are quasi-Toeplitz matrices. In fact, in this case, our method allows us to apply a quadratically convergent iteration where the current algorithms available in the literature, based on fixed point iterations \cite{bmmr:sisc}, \cite{bmm20}, only provide a linear convergence and require a much higher CPU time.
Finally, Section \ref{sec:exp} reports the results of some numerical experiments that demonstrate the effectiveness of our approach. We draw conclusions in Section~\ref{sec:con}.

\section{Preliminaries on SDA}\label{sec:prel}
We recall some properties of linear pencil at the basis of the design and analysis of SDA, we refer the reader to \cite[Chapter 5]{bim:book} and \cite{li:book} for more details. 

Let $M$, $N$ be $2n\times 2n$ matrices and consider the linear pencil $M-\lambda N$.
If 
\begin{equation}\label{ssf-i}
M=\begin{bmatrix}
E&0\\-P&I
\end{bmatrix},\quad
N=\begin{bmatrix}
I&-Q\\0&F
\end{bmatrix},
\end{equation}
where $E,F,P,Q$ and $I$ are $n\times n$ matrices, then the pencil is said to be in standard structured form of kind I (SSF-I).

Consider the problem of computing $n\times n$ matrices $X$ and $W$ such that
\begin{equation}\label{ssf-i-0}
M\begin{bmatrix}
I\\X
\end{bmatrix}=
N\begin{bmatrix}
I\\X
\end{bmatrix}W.
\end{equation}
We say that the columns of $\left[\begin{smallmatrix}I\\X\end{smallmatrix}\right]$ span a graph deflating subspace of the pencil $M-\lambda N$ associated with the eigenvalues of $W$.
If the pencil  is in SFF-I and satisfies \eqref{ssf-i-0},  then setting
\begin{equation}\label{ssf-i-k}
M_k=\begin{bmatrix}
E_k&0\\-P_k&I
\end{bmatrix},\quad
N_k=\begin{bmatrix}
I&-Q_k\\0&F_k
\end{bmatrix},
\end{equation}
with $E_0=E,~F_0=F,~P_0=P,~Q_0=Q$, and
\begin{equation}\label{eq:sda}
\begin{array}{ll}
 E_{k+1}=E_k(I-Q_kP_k)^{-1}E_k,& \quad P_{k+1}=P_k+F_k(I-P_kQ_k)^{-1}P_kE_k\\
 F_{k+1}=F_k(I-P_kQ_k)^{-1}F_k,&\quad Q_{k+1}=E_k(I-Q_kP_k)^{-1}Q_kE_k,
\end{array}
\end{equation}
for $k=0,1,\ldots$,
 yields the equation
\[
M_k\begin{bmatrix}
I\\X
\end{bmatrix}=
N_k\begin{bmatrix}
I\\X
\end{bmatrix}W^{2^k}.
\]
Here, we assume that the matrices $I-P_kQ_k$ and $I-Q_kP_k$ are invertible, for $k=0,1,\ldots$.

SDA consists in computing the sequences defined in \eqref{eq:sda} which, under suitable convergence properties, provide an approximation to the matrix $X$ for a sufficiently large values of $k$.

Denote $\rho(A)$ the spectral radius of the matrix $A$, i.e., the maximum modulus of the eigenvalues of $A$. We recall the following convergence results of SDA \cite[Theorems 5.3, 5.4]{bim:book}.

\begin{theorem}\label{th:sda}
If the scheme \eqref{eq:sda} can be carried out with no breakdown, and if $\rho(W)<1$,  $\|F_k\|<\gamma$ for some operator norm $\|\cdot\|$ and positive constant $\gamma$, then $\lim_k \|P_k-X\|=0$, and $\limsup_k \|X-P_k\|^{1/2^k}\le\rho(W)$.
\end{theorem}

The uniform boundedness of $\|F_k\|$ is guaranteed under the conditions expressed by the following 
\begin{theorem}\label{th:conv}
Let $X,Y,W,V$ be $n\times n$ matrices such that
\[
M\begin{bmatrix}
I\\X
\end{bmatrix}=
N\begin{bmatrix}
I\\X
\end{bmatrix}W,\quad
M\begin{bmatrix}
Y\\I
\end{bmatrix}V=
N\begin{bmatrix}
Y\\I
\end{bmatrix},
\]
and  $\sigma:=\rho(W)\le 1$, $\tau:=\rho(V)\le 1$, $\sigma\tau<1$. If the scheme \eqref{eq:sda} can be carried out with no breakdown, 
then 
\begin{equation}\label{eq:conv}
    \begin{array}{ll}
         E_k=(I-Q_kX)W^{2^k},&X-P_k=F_kXW^{2^k},  \\
         F_k=(I-P_kY)V^{2^k},&Y-Q_k=E_kYV^{2^k}. 
    \end{array}
\end{equation}
Moreover, we have $\limsup_k\|X-P_k\|^{1/2^k}\le\sigma\tau$, $\limsup_k\|Y-Q_k\|^{1/2^k}\le\sigma\tau$, $\limsup_k\|E_k\|^{1/2^k}\le\sigma$, $\limsup_k\|F_k\|^{1/2^k}\le\tau$.
\end{theorem}

We refer to the equation involving $Y$ and $V$ in Theorem \ref{th:conv} as to the {\em dual equation}.
Finally, we recall the following result \cite[Theorem 5.5]{bim:book} that gives conditions under which, given a general pencil $M-\lambda N$, there exists a pencil $\widehat M-\lambda\widehat N$ in SFF-I equivalent to $M-\lambda N$. We say that two pencils $M-\lambda N$ and $\widetilde M-\lambda \widetilde N$ are equivalent if there exist  nonsingular matrices $S_1$ and $S_2$ such that $\widetilde M=S_1MS_2$, $\widetilde N=S_1NS_2$.

\begin{theorem}\label{th:equiv}
Given a $2n\times 2n$ pencil $M-\lambda N$, partition the matrices $M$ and $N$ as $M=[M_1,M_2]$, $N=[N_1,N_2]$, where the blocks $M_i,N_i$ are $2n\times n$ matrices. If $S=[N_1,M_2]$ is invertible, then 
\[
\begin{bmatrix}
E&-Q\\-P&F
\end{bmatrix}=S^{-1}[M_1,N_2]
\]
defines, through \eqref{ssf-i}, a pencil in SSF-I that is equivalent to $M-\lambda N$. Moreover the invertibility of $S$ is a necessary condition.
\end{theorem}

\section{Solving the quadratic matrix equation}\label{sec:lin}
Here and hereafter, we assume that $A_0$ is invertible and that there exist matrices $V$ and $G$ such that $\rho(G)\le 1$, $\rho(V)\le 1$, $\rho(G)\rho(V)<1$ and
\begin{equation}\label{eq:ass}
A_1G^2+A_0G+A_{-1}=0,\quad A_{-1}V^2+A_0V+A_1=0.
\end{equation}
The goal is to compute $G$, given an approximation $\widetilde G$.

\subsection{Linearization of the quadratic matrix equation}
The quadratic matrix equations \eqref{eq:ass} can be equivalently rewritten in terms of invariant subspaces as
\begin{equation}\label{eq:comp}
    M \begin{bmatrix}
    I\\G
    \end{bmatrix}=N\begin{bmatrix}
    I\\G
    \end{bmatrix}G,~~~
 M \begin{bmatrix}
    V \\I 
    \end{bmatrix} V =N\begin{bmatrix}
    V \\I
    \end{bmatrix},
 \end{equation}
where
\begin{equation}\label{eq:compmn}
M=
    \begin{bmatrix}
    0&I\\-A_{-1}&-A_0
    \end{bmatrix},\quad N=
    \begin{bmatrix}
    I&0\\ 0&A_1
    \end{bmatrix}.
\end{equation}

In view of Theorem \ref{th:equiv}, we may reduce the pencil $M-\lambda N$, with $M$ and $N$ defined in \eqref{eq:compmn}, into a pencil $\widehat M-\lambda \widehat N$ in SSF-I, where
\begin{equation}\label{eq:mhat}
\widehat M=
    \begin{bmatrix}
    -A_0^{-1}A_{-1}&0\\A_0^{-1}A_{-1}&I
    \end{bmatrix},\quad \widehat N=
    \begin{bmatrix}
    I&A_0^{-1}A_1\\ 0&-A_0^{-1}A_1
    \end{bmatrix}.
\end{equation}
We may easily verify that
\begin{equation}\label{eq:is1}
\widehat M \begin{bmatrix}
    I\\G
    \end{bmatrix}=\widehat N\begin{bmatrix}
    I\\G
    \end{bmatrix}G,~~~
\widehat M\begin{bmatrix}
V\\I
\end{bmatrix}V=\widehat N\begin{bmatrix}V\\I\end{bmatrix}.
   \end{equation}
This way, we may apply SDA in order to solve equation \eqref{eq:mateq}, that is, apply \eqref{eq:sda} with $E_0=P_0=-A_0^{-1}A_{-1}$ and
$F_0=Q_0=-A_0^{-1}A_{1}$.
Therefore, since $\rho(G)\rho(V)<1$, by applying Theorem \ref{th:conv} we find that $P_k$ converges to $G$, $Q_k$ converges to $V$. Moreover, $\limsup_k \| P_k-G\|^{1/2^k}\le\rho(G)\rho(V)$, $\limsup_k \| Q_k-V\|^{1/2^k}\le \rho(G)\rho(V)$.

Observe also that, unlike fixed point iterations,  the SDA in the form \eqref{eq:conv} does not allow to refine a given initial approximation to $G$ and to $V$. In the next section we overcome this drawback.

\subsection{Defect equation for the quadratic matrix equation}\label{sec:sdaref}
Here we follow the lines of \cite{defectBG}, \cite{defectM} where the technique of defect-correction is introduced for refining an approximation to the solution of a discrete-time algebraic Riccati equation.

Assume that we are given an approximation $\widetilde G$ to $G$, and write
\begin{equation}\label{eq:tg}
G=\widetilde G+H.
\end{equation}
Replacing \eqref{eq:tg} in the first equation of \eqref{eq:ass} yields
\begin{equation}\label{eq:mateq1}
\begin{aligned}
&A_1H^2+A_1H\widetilde G+(A_0+A_1\widetilde G)H+R(\widetilde G)=0,\\
&R(\widetilde G)=A_1\widetilde G^2+A_0\widetilde G+A_{-1},
\end{aligned}
\end{equation}
where now the unknown is $H$.

By following the lines described at the beginning of Section \ref{sec:lin}, we can rewrite \eqref{eq:mateq1} in terms of an invariant subspace formulation as follows
\begin{equation}\label{eq:is2}
\widetilde M\begin{bmatrix}
I\\ H
\end{bmatrix} = \widetilde N\begin{bmatrix}
I\\ H
\end{bmatrix}G,\qquad G=H+\widetilde G,
\end{equation}
where
\begin{equation}\label{eq:mtilde}
\widetilde M=
\begin{bmatrix}
\widetilde G&I\\-R(\widetilde G)&-(A_0+A_1\widetilde G)
\end{bmatrix},~~~ \widetilde N=N.
\end{equation}

We may easily verify that
\[
\widetilde M=\begin{bmatrix}
I & 0 \\ -A_1\widetilde G & I
\end{bmatrix}
M
\begin{bmatrix}
I & 0\\
\widetilde G & I
\end{bmatrix},~~~
\widetilde N=\begin{bmatrix}
I & 0 \\ -A_1\widetilde G & I
\end{bmatrix}
N
\begin{bmatrix}
I & 0\\
\widetilde G & I
\end{bmatrix}.
\]
This fact leads to the following

\begin{lemma}
The pencils  $M-\lambda N$ and $\widetilde M-\lambda \widetilde N$, with $M$, $N$ and $\widetilde M$, $\widetilde N$ defined in \eqref{eq:compmn} and \eqref{eq:mtilde}, respectively, are equivalent.
\end{lemma}

We may easily verify that
\begin{equation}\label{eq:is3}
\widetilde M\begin{bmatrix}
Y \\ I
\end{bmatrix} Z = \widetilde N\begin{bmatrix}
Y \\ I
\end{bmatrix},\quad Y=V(I-\widetilde G V)^{-1},~~Z=(I-\widetilde G V)V(I-\widetilde G V)^{-1}.
\end{equation}
By applying Theorem \ref{th:equiv}, we can transform the pencil $\widetilde M-\lambda \widetilde N$ into the pencil $\Check M-\lambda\Check N$ in SSF-I, where
\begin{equation}\label{eq:shifted}
    \Check M=\begin{bmatrix}
        -SA_{-1}&0\\SR(\widetilde G)&I
    \end{bmatrix},\quad
    \Check N=\begin{bmatrix}
        I&SA_1\\ 0&-SA_1
    \end{bmatrix},\quad S=(A_0+A_1\widetilde G)^{-1}.
\end{equation}
By construction, we have
\begin{equation}\label{eq:is4}
\Check{M}\begin{bmatrix}
    I\\ H
    \end{bmatrix}=\Check N\begin{bmatrix}
    I\\ H
    \end{bmatrix}G,~~
    \Check M \begin{bmatrix}
Y\\ I
\end{bmatrix}Z=\Check N\begin{bmatrix}
Y\\I\end{bmatrix}.
\end{equation}

This way, the increment $H$ can be viewed in terms of a graph invariant subspace so that it can be computed, say, by means of SDA. This is the subject of the next section.

\subsection{Defect correction algorithm based on SDA}\label{sec:dcsda}
In this section, we provide a variant of the SDA which allows us to refine the initial approximation $\widetilde G$, by computing $H$.

The SDA iteration applied to \eqref{eq:is4} for computing the solution $H$ of equation \eqref{eq:mateq1} consists in setting
\begin{equation}\label{eq:sdashifted}
    \begin{array}{ll}
    P_0=-(A_0+A_1\widetilde G)^{-1}R(\widetilde G),& E_0=\widetilde G+P_0, \\
    F_0=-(A_0+A_1\widetilde G)^{-1}A_1, & Q_0=-F_0,\\
    \end{array}
\end{equation}
and in applying equations \eqref{eq:sda}.

Since $V$ and $Z$ have the same eigenvalues, then $\rho(G)\rho(Z)<1$, therefore, in view of Theorem \ref{th:conv}, we may conclude that SDA applied to  \eqref{eq:is4}  is convergent. Moreover, equations \eqref{eq:conv} turn into
\begin{equation}\label{eq:conv1}
    \begin{array}{ll}
         E_k=(I-Q_kH)G^{2^k},&H-P_k=F_kH G^{2^k},  \\
         F_k=(I-P_kY)Z^{2^k},&Y-Q_k=E_kYZ^{2^k}.
    \end{array}
\end{equation}
Since $P_k$ converges to $H$, at each step $k$ we get an approximation to the solution $G$ in the form $G_k=\widetilde G+P_k$.

At a first glance, equations \eqref{eq:conv1} seem to provide no substantial advantage in the acceleration to the convergence with respect to the analogous equations \eqref{eq:conv}, applied to the original matrix equation \eqref{eq:mateq}, with $W=G$ . In fact, in both cases, the convergence speed is determined by the factor $\sigma\tau$ for $\sigma=\rho(G)$ and $\tau=\rho(V)$. The only difference seems to be that in the upper bound to the norm of the error, that is $\|H-P_k\|\le \|F_k\|\,\|H\|\,\|G^{2^k}\|$, the factor  $\|H\|$ is smaller the closer is the initial approximation $\tilde G$ to the solution $G$.  But a more accurate analysis shows that the acceleration may be substantial as shown in the next section.

\section{The stochastic case}\label{sec:stoc}
Consider the case where $A_i=-B_1$, for $i=-1,1$, and $A_0=I-B_0$, where $B_i\ge 0$, $i=-1,0,1$, and $(B_{-1}+B_0+B_1)\one=\one$, where $\one=(1,1,\ldots,1)^T$.
Under this assumption, there exist  unique minimal nonnegative solutions $G$ and $V$ to \eqref{eq:ass}, respectively.

The matrices $B_i$, $i=-1,0,1$, define the homogeneous part of the transition matrix of a Quasi-Birth-and-Death process \cite{blm:book}. If such Markov chain is positive recurrent, then $G\one =\one$, so that $\rho(G)=1$ and $\rho(V)<1$. If the Markov chain is transient, then $V\one =\one $ and $\rho(G)<1$, while if the Markov chain is null recurrent then $G\one =V\one=\one$, so that $\rho(G)\rho(V)=1$.

Throughout this section we assume recurrence that is, $G\one=\one$. We also assume, without loss of generality, that $\lambda=1$ is the only eigenvalue of modulus 1 of $G$ \cite{blm:book}. The case $V\one=\one$ can be treated by exchanging the roles of $A_{-1}$ and $A_1$.

\subsection{Convergence acceleration}\label{sec:ca}
Observe that, if $\widetilde G$ is a 
stochastic matrix, then $H\one=G\one-\widetilde G \one=0$.
Therefore,  in view of \eqref{eq:conv1}, we have $(H-P_k)\one=F_k H \one=0$.
This implies that the matrices $P_k$ are stochastic for any $k$ and that there is no error in the approximation $P_k$ to  $H$ along the direction given by $\one$.
Moreover, if $u$ is any other eigenvector of $G$ corresponding to an eigenvalue $\lambda$ different from 1, we have 
$F_kH G^{2^k}u=\lambda^{2^k}F_kH u$. This implies that, if $G$ has $n$ linearly independent eigenvectors corresponding to the eigenvalues $\lambda_i$, $i=1,\ldots,n$ such that  $1=\lambda_1>|\lambda_2|\ge\cdots\ge |\lambda_n|$, then there exists a constant $\gamma$ such that for any vector $w$ one has $\|H G^{2^k}w\|\le \gamma|\lambda_2|^{2^k} \|w\|$. That is $\|H-P_k\|\le\tilde \gamma (\tau\tilde\sigma)^{2^k}$, where $\tilde \gamma$ is a constant and $\tilde\sigma=|\lambda_2|$ is the second largest eigenvalue of $G$ in modulus. This implies that $\limsup_k \|H-P_k\|^{1/2^k}\le \tau\tilde\sigma<\tau\sigma$. The same conclusion can be obtained in the case where $G$ has nontrivial Jordan blocks. 

This actually provides a strong acceleration especially in the cases where $\tau$ is close to 1, i.e., the stochastic process is close to be null-recurrent, and still guarantees superlinear convergence if the process is null recurrent, i.e., $\rho(V)=\rho(G)=1$.

Another observation concerning \eqref{eq:conv1} is that the expression for $Y-Q_k$ can be rewritten as 
\[
Y-Q_k=E_kV^{2^k+1}(I-\widetilde GV)^{-1}.
\]

It is worth recalling that a similar acceleration has been obtained in \cite[Section 2.6]{bim:book} by modifying the original equation \eqref{eq:mateq} into a new equation $\widetilde A_1X^2+\widetilde A_0 X+\widetilde A_{-1}=0$ whose solution differs from $G$ by a stochastic rank-1 correction. This manipulation is performed in such a way to shift the eigenvalue 1 of the pencil $M-\lambda N$ to zero.
However, the approach that we have introduced seems to be more general than the one shown in \cite[Section 2.6]{bim:book} since unlike the latter technique, it allows to choose any  initial approximation as $\widetilde G$, not necessarily a stochastic rank-1 matrix.

\subsection{A further improvement}\label{sec:fi}
Write the stochastic approximation $\widetilde G$ as $\widetilde G=\one u^T +S$, where $u$ is any vector such that $u^T \one=1$, and $S \one=0$.
This way, the matrix $H=G-\widetilde G$ is such that $H\one =0$ and $H\widetilde G=HS$. This property implies that equation \eqref{eq:mateq1} can be simplified into
\begin{equation}\label{eq:simple}
    A_1 H^2 + A_1 HS +(A_0+A_1\widetilde G)H+R(\widetilde G)=0.
\end{equation}
Such equation can be reformulated in terms of invariant subspaces as
\begin{equation}\label{eq:imp}
\xbar{M}     \begin{bmatrix} I \\ H
    \end{bmatrix}=\xbar{N}  \begin{bmatrix} I \\ H
    \end{bmatrix}(S+H)
\end{equation}
where $\xbar{M}$ and $\xbar{N}$ are defined by
\begin{equation}\label{eq:barm}
 \xbar{M}=   \begin{bmatrix}
    S & I \\
    -R(\widetilde G) & -(A_0+A_1\widetilde G)
    \end{bmatrix},\quad
    \xbar{N}=
    \begin{bmatrix}
    I & 0 \\
    0 & A_1
    \end{bmatrix}.
\end{equation}
Observe that $S+H=G-\one u^T$, so that the eigenvalues of $S+H$ coincide with the eigenvalues of $G$, except for $\lambda=1$, which is replaced by 0, in view of the Brauer Theorem \cite{bra}.

The pencil $\xbar{M}-\lambda\xbar{N}$ in \eqref{eq:barm} can be reduced in SSF-I by applying Theorem \ref{th:equiv}, this way we get
\begin{equation}\label{eq:impssf}
\begin{aligned}
&    \begin{bmatrix}
    S-KR(\widetilde G)&0\\ KR(\widetilde G)&I
    \end{bmatrix}
    \begin{bmatrix}I \\ H
    \end{bmatrix}=
    \begin{bmatrix}
        I&KA_1\\ 0&-KA_1
    \end{bmatrix}
        \begin{bmatrix}I \\ H
    \end{bmatrix}(S+H),\\
&    K=(A_0+A_1\widetilde G)^{-1}.
\end{aligned}
\end{equation}
Therefore, the SDA iteration \eqref{eq:sda} can be applied with the starting values 
\[
E_0=S-KR(\widetilde G), ~~F_0=Q=-KA_1,~~P_0=-KR(\widetilde G).
\]

The following result is fundamental for the convergence of SDA.

\begin{proposition}\label{prop}
Let $\widetilde M$ and $\widetilde N$ be the matrices defined in \eqref{eq:is2} and $\xbar{M}$, $\xbar{N}$ the matrices defined in \eqref{eq:barm}. Set $\widetilde \phi(\lambda)=\widetilde M -\lambda \widetilde N$,
$\xbar{\phi}(\lambda)=\xbar{M} -\lambda \xbar{N}$. Then  
\[
\xbar{\phi}(\lambda)=\widetilde \phi(\lambda)\left(I-\lambda^{-1}C\right)^{-1},\quad C=\begin{bmatrix}
    \one \\ 0 \end{bmatrix} \begin{bmatrix}
    u^T & 0^T
\end{bmatrix},
\]
and the eigenvalues of $\xbar\psi(\lambda)$ are the eigenvalues of  $\widetilde \psi(\lambda)$, except for the eigenvalue 1 which is replaced by 0.
Moreover, if the solution $V$ of \eqref{eq:ass} is diagonalizable, i.e., $V=S D S^{-1} $, $D =\diag(\mu_1,\ldots,\mu_n)$,   then 
\[
\xbar{M} \begin{bmatrix}
    \widehat Y \\ I
\end{bmatrix} Z = \xbar{N} \begin{bmatrix}
    \widehat Y \\ I
\end{bmatrix}
\]
where $\widehat Y=Y-D \one u^T Y$.
\end{proposition}

\begin{proof}
The equation relating $\xbar{\phi}(\lambda)$ and $\widetilde\phi(\lambda)$  follows from the property
\[
\widetilde M \begin{bmatrix}
    \one \\ 0 \end{bmatrix}=\begin{bmatrix}
    \one \\ 0 \end{bmatrix},~~\widetilde N \begin{bmatrix}
    \one \\ 0 \end{bmatrix}=\begin{bmatrix}
    \one \\ 0 \end{bmatrix}
\]
and from \cite[Theorem 3.32]{blm:book}. Concerning the second part, 
from \eqref{eq:is3} we deduce that
\[
\widetilde \phi(\lambda)
\begin{bmatrix}
    Y \\ I
\end{bmatrix}=\widetilde M \begin{bmatrix}
    Y \\ I
\end{bmatrix} (I - \lambda Z).
\]
For the sake of simplicity assume that the eigenvalues $\mu_i$ of $V$ are nonzero. Since 
$\widetilde \phi(\lambda)=\xbar{\phi}(\lambda)(I-\lambda^{-1}C)$, we obtain
\[
\xbar{\phi}(\lambda)
\begin{bmatrix}
    Y -\lambda^{-1}\one u^T Y\\ I
\end{bmatrix}=\widetilde M \begin{bmatrix}
    Y \\ I
\end{bmatrix} (I - \lambda Z).
\]
By multiplying to the right by 
$u_i=(I-\widetilde G V)S e_i$, where $e_i$ is the $i$-th column of the identity matrix, and by choosing $\lambda=\mu_i^{-1}$, since $(I - \mu_i^{-1} Z)u_i=0$, we obtain that
\[
\widetilde \phi(\mu_i^{-1}) \begin{bmatrix}
    (Y -\mu_i \one u^T Y)u_i \\ u_i
\end{bmatrix}=0, ~~i=1,\ldots,n.
\]
These latter equations imply that the columns of the matrix
\[
\begin{bmatrix}
    (Y-D\one u^TY)(I-\widetilde G V)S \\
    (I-\widetilde G V)S
\end{bmatrix}
\]
span the invariant subspace of $\xbar{\phi}(\lambda)$ corresponding to the eigenvalues of modulus greater than 1.
The proof is completed by setting $\widehat Y=Y-D \one u^T Y$.
Since $\widehat Y$ does not involve the reciprocals of $\mu_i$, the assumption $\mu_i\ne 0$ can be relaxed.
\end{proof}

Under the assumptions of Proposition \ref{prop}, 
if the SDA process has no break-down, then the hypotheses of Theorem \ref{th:conv} are satisfied, and  equations \eqref{eq:conv} turn into
\begin{equation}\label{eq:conv2}
    \begin{array}{ll}
         E_k=(I-Q_kH)G^{2^k},&H-P_k=F_kH (S+H)^{2^k},  \\
         F_k=(I-P_kT)W^{2^k},&T-Q_k=E_kYW^{2^k}.
    \end{array}
\end{equation}
Moreover, $\limsup_{k}\sqrt[2^k]{\Vert H-P_k \Vert}\le \tau\tilde\sigma<\tau\sigma $, where $\tilde\sigma$ is the second largest modulus eigenvalue of $G$.

\section{An application to quasi-Toeplitz matrices}\label{sec:qt}
In this section, we show that the convergence properties of the  algorithms presented in the previous sections, still hold in the case where the matrices $A_{-1},A_0,A_1$ and $\widetilde G$ belong to an (infinite dimensional) Banach
algebra $\mathcal B$, i.e., an algebra endowed with an operator norm that
makes it a Banach space.

\subsection{SDA in a Banach algebra}
The convergence results recalled in Section \ref{sec:prel} have been proved in \cite[Chapter 5]{bim:book} by relying on the properties of matrix algebras endowed with any operator norm $\|\cdot\|$ and on the following two properties, valid for any $n\times n$ matrix $A$:
\begin{enumerate}
    \item $\lim_k\|A^k\|^{1/k}=\rho(A)$;
    \item if $\rho(A)<1$, then there exist $\beta>0$, $\rho(A)<\sigma<1$, such that $\|A^k\|\le\beta\sigma^k$ for any $k\ge 0$.
\end{enumerate}
Therefore, in order to extend the validity of convergence results of SDA to the case of matrices belonging to a Banach algebra $\mathcal B$, we have to show that the above two properties still hold if $A$ belongs to 
$\mathcal B$. 

 We recall that, for any $A\in\mathcal B$, the spectral radius is defined as $\rho(A)=\sup\{\vert\lambda\vert:~\lambda I-A \hbox{ not invertible in }\mathcal B\}$. Actually, the property $\lim_k\|A^k\|^{1/k}=\rho(A)$, also known as Gelfand theorem,  holds true on any Banach algebra $\mathcal B$ independently of the finiteness of the matrix size  \cite{bhatia:book}. Moreover, the second property is a direct consequence of Gelfand theorem. In fact, from the
first property on finds that for any $\epsilon>0$ there exists $k_0>0$
such that  $|\rho(A)-\|A^k\|^{1/k}|\le \epsilon$ for any $k\ge k_0$.
The latter inequality can be rewritten as $(\rho(A)-\epsilon)^k\le \|A^k\|\le (\rho(A)+\epsilon)^k$. 
Choose $\epsilon>0$ small enough so that $\sigma:=\rho(A)+\epsilon<1$, this way, the
condition $\|A^k\|\le\sigma^k$ is satisfied for any $k\ge k_0$.
Then define $\beta=\max\{1,\frac{\max_{1\le k<k_0}\|A^k\|} {(\rho(A)+\epsilon)^{k_0}}\}$ 
and find that, for $1\le k<k_0$, 
\[
\beta\sigma^k\ge \frac{\|A^k\|}{(\rho(A)+\epsilon)^{k_0}}(\rho(A)+\epsilon)^k\ge \|A^k\|.
\]
On the other hand, since $\beta\ge 1$ we have $\|A^k\|\le\sigma^k\le \beta\sigma^k$ for $k\ge k_0$.

Therefore, Theorems~\ref{th:conv} and \ref{th:sda} are still valid in the the case where $E$, $P$, $F$, $Q$, $X$, $Y$, $W$, and $V$  belong to a given Banach algebra $\mathcal B$. 
In particular, the convergence properties stated in Section~\ref{sec:lin} hold in the case where $A_{-1},A_0,A_1\in\mathcal B$, provided that the quadratic matrix equations \eqref{eq:ass} have solutions $G,V\in\mathcal B$ such that $\rho(G),\rho(V)\le 1$ and $\rho(G)\rho(V)<1$. 

An interesting Banach algebra is given by the (extended) quasi-Toeplitz matrices.

\subsection{Quadratic equations in the extended quasi-Toeplitz algebra}
The Extended Quasi-Toeplitz  class $\mathcal{EQT}$ \cite{bmmr:sisc} is a Banach algebra, endowed with the infinity norm, formed by matrices of the kind
\[
A=T(a)+E_a+\one v^T,
\]
where $T(a)=(t_{i,j})$, $t_{i,j}=a_{i-j}\in\mathbb R$ for $i,j\in\mathbb Z^+$, $\sum_{k\in\mathbb Z} |a_k|<\infty$ is a Toeplitz matrix, $E_a=(e_{i,j})$, $e_{i,j}\in\mathbb R$, is such that $\|E_a\|_\infty<+\infty$, and satisfies the decay property $\lim_{i\to\infty}\sum_{j=1}^\infty |e_{i,j}|=0$ so that it is a compact operator, and $v=(v_i)$, $v_i\in\mathbb R$ is a vector in $\ell^1$, that is, $\sum_{i=1}^\infty |v_i|<+\infty$. We refer to $T(a)$, $E_a$ and $\one v^T$ as the Toeplitz part, the compact correction, and the limit part of $A$, respectively. Moreover, we denote by quasi-Toeplitz, the subalgebra $\mathcal{QT}\subset\mathcal{EQT}$ formed by matrices having limit part equal to zero \cite{bmmr:sisc}.  

Quadratic matrix equations with coefficients in $\mathcal{QT}$ are encountered in the analysis of random walks in the quarter plane
performed with the matrix-geometric methodology of \cite{neuts}, that can be applied by looking at any bidimensional random walk as a Quasi-Birth-and-Death stochastic process \cite{lr:book}. The random walk model analyses the dynamic of a particle that can occupy the points $(i,j)$ of a grid in the quarter plane $i,j\ge 0$. The particle in position $(i,j)$ can move to the 9 neighborhood positions $(i+\delta_i,j+\delta_j)$ for $\delta_i,\delta_j\in\{-1,0,1\}$, with given (known) probabilities $a_{\delta_i,\delta_j}$ independent of the position $(i,j)$ if $i,j>0$ (inner part). For $j=0$ the probabilities are given by the $3\times 2$ matrix $(x_{\delta_i,\delta_j})$ for $\delta_i\in\{-1,0,1\}$ and $\delta_j\in\{0,1\}$. Similarly, for $j=0$ the probabilities are given by the $2\times 3$ matrix $(y_{\delta_i,\delta_j})$ for $\delta_i\in\{0,1\}$ and $\delta_j\in\{-1,0,1\}$, and for $i=j=0$ the probabilities are given by the $2\times 2$ matrix $o_{\delta_i,\delta_j}$, for $\delta_i,\delta_j\in\{0,1\}$. Figure~\ref{fig:qp} pictorially summarizes the dynamic of this system.
 For more details we refer the reader to \cite{bmmr}, \cite{bmmr:sisc}, \cite{ozawa}. 
 
 \begin{figure}
     \centering
     \includegraphics[scale=0.4]{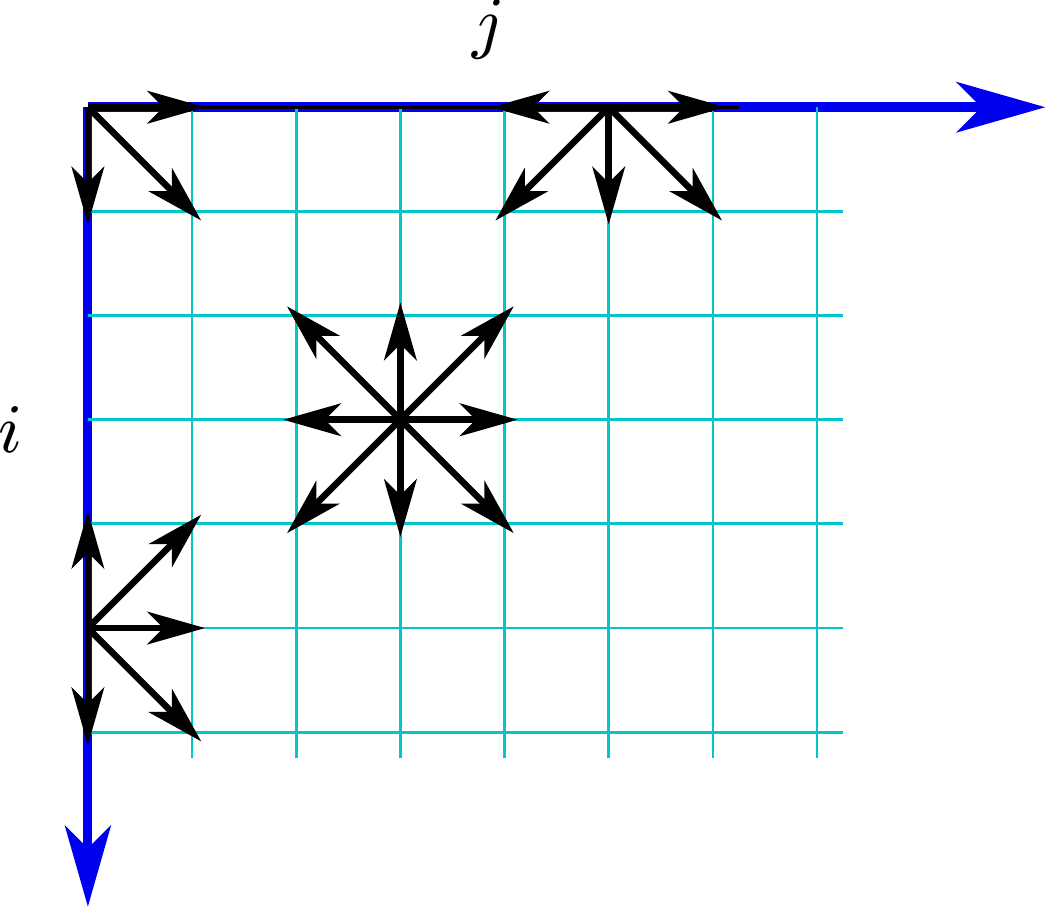}
     \caption{Dynamics of a random walk in the quarter plane}
     \label{fig:qp}
 \end{figure}
 
 The analysis of this model leads to the quadratic matrix equation \eqref{eq:mateq} where the coefficients $A_i$, are such that $A_i=I-B_i$, $i=-1,0,1$, and $B_i$ are tridiagonal matrices with the Toeplitz structure with entries $[a_{i,-1},a_{i,0},a_{i,1}]$ except for the entries in the first row which are determined by the compact correction, that is, $B_i=T(a^{(i)})+E_i$, where $a^{(i)}=[a_{i,-1},a_{i,0},a_{i,1}]$, and $E_i$ has null entries except the first row, that coincides with $[x_{i,0}-a_{i,0},x_{i,1}-a_{i,1},0,\ldots]$. 
 
 Under very mild conditions, there exists the {\em minimal nonnegative solution} $G$ of the matrix equation \eqref{eq:mateq}, where minimal is with respect to the component-wise ordering. If $\sum_{j=-1}^1 a_{-1,j}>\sum_{j=-1}^1 a_{1,j}$, then this solution belongs to $\mathcal{QT}$, that is, its limit part is zero. Whereas, if $\sum_{j=-1}^1 a_{-1,j}<\sum_{j=-1}^1 a_{1,j}$, and  $G$ is stochastic, then $G$ belongs to $\mathcal{EQT}\setminus\mathcal{QT}$, that is, its limit part is nonzero \cite{bmmr:sisc}.
  
 This latter case is difficult to handle. In fact, by applying the standard SDA iteration  with coefficients $A_i\in \mathcal{QT}$, since $\mathcal{QT}$ is an algebra, all the matrices $E_k,F_k,P_k$ and $Q_k$, generated by the algorithm belong to $\mathcal{QT}$, whereas the solution $G$ of the matrix equation is such that $G\in\mathcal{QT}\setminus\mathcal{EQT}$. Whence the sequence $\{P_k\}$ cannot converge in $\mathcal{QT}$. In practice, the compact correction part of $P_k$ should approximate also the limit part, that has columns proportional to the infinite vector $\one$ of components equal to 1.
 From the theoretical point of view this fact implies that
 there exists a constant $\gamma>0$ such that $\|P_k-G\|_\infty\ge \gamma$ independently of $k$.
 Computationally, this produces a very strong slow-down and a memory overflow after a few steps. The same drawback holds also for cyclic reduction and logarithmic reduction, while for fixed point iterations, including Newton's iteration, this can be overcome by choosing as starting approximation a suitable matrix in $\mathcal{EQT}\setminus\mathcal{QT}$.
 
 On the other hand, applying the modified version of the SDA of Section~\ref{sec:sdaref} and by choosing, say, 
 \begin{equation}\label{eq:T1}
 \widetilde G=\frac12 I+\frac12 \one v^T\in \mathcal{EQT},
  \end{equation}
  where $v$ is any vector such that 
$v^T\one=1$, e.g., $v^T=[1,0,0,\ldots]$,
 has the following advantages:
 \begin{enumerate}
     \item  the initial matrices $E_0,F_0,P_0,Q_0$ defined in \eqref{eq:sdashifted} belong to $\mathcal{EQT}\setminus\mathcal{QT}$, so that the computation is maintained in $\mathcal{EQT}\setminus\mathcal{QT}$ where the solution lives;
     \item the compact correction and the limit correction of $P_k$ converge to the compact correction and to the limit correction  of $H=G-\widetilde G$, respectively; this follows from the fact that $\lim_k\|P_k-H\|_\infty=0$;
     \item the convergence is generally faster than the standard SDA. 
 \end{enumerate}
 In particular, due to the decay property of the compact correction, and to the fact that the vector $v$ in the limit correction is in $\ell^1$, property 2 implies that, for any given $\epsilon>0$, the set of entries in the compact correction of $P_k$ and in the vector $v^T$ of the limit correction  having modulus greater than or equal to $\epsilon$ is finite. This fact overcomes the problem of memory overflow and slowdown.

A further advantage can be obtained by choosing $\widetilde G$ as follows
\begin{equation}\label{eq:gtg}
\widetilde G=T(g)+(\one-T(g)\one)e_1^T,
\end{equation}
where $T(g)$ is the Toeplitz part of the sought solution $G$ and $e_1=[1,0,0,\ldots]^T$.  In this way, the matrix $\widetilde G$ is stochastic,  and its Toeplitz part coincides with the Toeplitz part of $G$.
With this choice, the Toeplitz part of the matrices $P_k$ is constantly zero, this simplifies part of the iteration.
Moreover, as shown in \cite{bmmr:sisc}, the computation of $T(g)$ is very cheap, it can be computed once for all,  and does not affect the overall cost of the procedure.

The next section shows the actual advantages of this approach.

\section{Numerical experiments}\label{sec:exp}
We have tested our algorithm on matrix equations coming from the modeling of random walks in the quarter plane, as described in Section~\ref{sec:qt}. The coefficients of the quadratic matrix equation \eqref{eq:mateq} are $A_{-1}=-B_{-1}$, $A_0=I-B_0$, $A_1=-B_1$, where $B_{-1},B_0,B_1$ are $\mathcal{QT}$ matrices, having the leading $2\times 3$ submatrix defined as follows:

\begin{description}
\item[Test 1]
{\small\[
B_{-1}=\frac19\begin{bmatrix}
3&3&0\\ 2&0&1
\end{bmatrix},\quad 
B_{0}=\frac19\begin{bmatrix}
1&1&0\\ 1&0&1
\end{bmatrix}, \quad 
B_{1}=\frac19\begin{bmatrix}
0&1&0\\ 2&1&1
\end{bmatrix}.
\]}
\item[Test 2]
{\small\[
B_{-1}=\frac1{16}\begin{bmatrix}
5&5&0\\ 2&0&1
\end{bmatrix},\quad 
B_{0}=\frac1{16}\begin{bmatrix}
2&2&0\\ 7&0&2
\end{bmatrix}, \quad 
B_{1}=\frac1{16}\begin{bmatrix}
1&1&0\\ 2&1&1
\end{bmatrix}.
\]}
\item[Test 3]
{\small\[
B_{-1}=\alpha\begin{bmatrix}
484&121&0\\ 80&120&160
\end{bmatrix},~
B_{0}=\alpha\begin{bmatrix}
121&0&0\\ 84&80&80
\end{bmatrix}, ~
B_{1}=\alpha\begin{bmatrix}
121&121&0\\ 160&124&80
\end{bmatrix},
\]}where $\alpha=\frac1{968}$.
\end{description}
Tests 1 and 2 are two examples reported in \cite{bmmr:sisc}, while the third test has been designed so that the solution $G$ of the matrix equation \eqref{eq:mateq} has the Toeplitz part with a very large numerical bandwidth and the correction part with a very large number of entries  greater than the machine precision (see Table~\ref{tab:g} and Figure~\ref{fig:g}). This way, computing the solution $G$ for the matrix equation of Test 3 is expected to be more difficult than computing the solution of Tests 1 and 2. Moreover, for all these tests, the solution $G$ has a nonzero limit part. 
For the sake of completeness, in Figure \ref{fig:g} we display in log scale the graphics of a portion of the computed solutions of Tests 1,2 and 3,  namely the  leading principal submatrix of size $1000\times 1000$. We may see the different decay to zero of the entries in the three cases.

\begin{figure}
    \centering
    \includegraphics[scale=0.4]{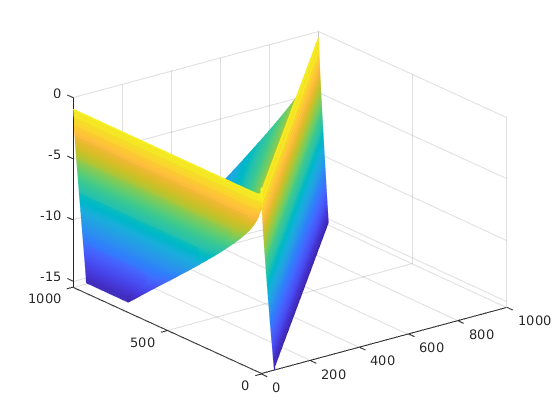}~ \includegraphics[scale=0.4]{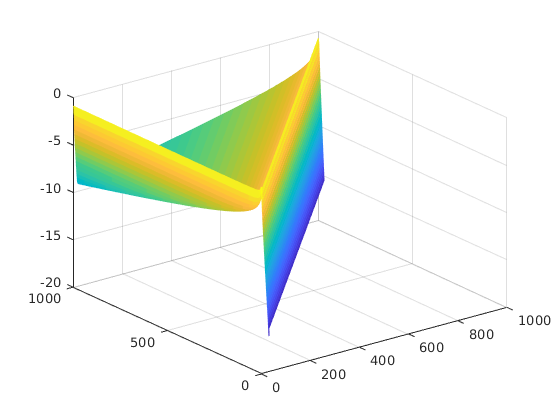}\\
    \includegraphics[scale=0.4]{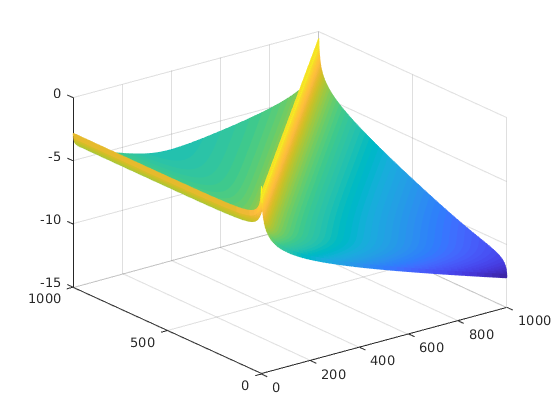}~
    \caption{\footnotesize Log-scale plot of the solution $G$ of Tests 1, 2 (first row) and 3 (second row). Only the portion $G(1:1000,1:1000)$ is displayed. The three matrices have a very different decay rate.}
    \label{fig:g}
\end{figure}

 One may expect that an effective way to solve equation \eqref{eq:mateq} in the case of matrix coefficients with infinite size is truncating these coefficients to a sufficiently large finite size $n$ and solving the finite size equation obtained this way. But in general, the solution of the finite equation is not a good approximation to the solution $G$ of the original equation. In fact, as shown in \cite{latouche02}, the solution of the finite equation strongly depends on the way the loss of stochasticity consequent to the finite truncation  is fixed. For this reason, it is crucial to solve the original equation in the infinite environment where the solution lives.

The tests have been run on a machine with an
Intel(R) Xeon(R) W-2145 CPU @ 3.70GHz with 8 cores, with the Linux operating system Linux Debian 5.10.0-0.deb10.16-amd64, using
Matlab version 9.10.0.1710957 (R2021a).

The computation relies on the toolbox {\tt CQT-Toolbox} of \cite{bmr:cqt} that implements the standard matrix operations in the algebra $\mathcal{EQT}$, the precision of computation has been set to $10^{-15}$.
The code is available from the authors upon request.

Recall that the modified SDA, introduced in Section~\ref{sec:ca},  with the choice $\widetilde G=0$ coincides with the classical SDA iteration described in Section~\ref{sec:prel}. In the case of the Tests 1, 2, and 3, the classical SDA iteration does not provide a convergent sequence since the solution $G$ has a nonzero limit part, i.e., $G\in \mathcal{EQT}\setminus\mathcal{QT}$, while all the approximations generated by classical SDA have a zero limit part since they belong to $\mathcal{QT}$.
In fact, to overcome this problem, in \cite{bmmr:sisc} a suitable fixed point iteration is proposed.

We have implemented 
and tested the new SDA variants with two different choices of $\widetilde G$, namely, 
$\widetilde G=\widetilde G_1:=\frac12(I+\one e_1^T)$ given in \eqref{eq:T1}, and $\widetilde G=\widetilde G_2:=T(g)+(\one -T(g)\one )e_1^T$ of \eqref{eq:gtg}. 
In both cases, $\widetilde G$ is a stochastic matrix belonging to $\mathcal{EQT}\setminus\mathcal{QT}$.
From the numerical experiments, the two variants proposed in Section~\ref{sec:dcsda} and \ref{sec:fi}, respectively, have comparable performances, therefore we report only the results concerning the SDA variant described in Section~\ref{sec:dcsda}. 
The two SDA iterations obtained with $\widetilde G=\widetilde G_1$ and $\widetilde G=\widetilde G_2$  are denoted by 
$SDA_1$, $SDA_2$, respectively.

We have also implemented the fixed point iteration 
$X_{k+1}=F(X_k)$ suggested in  \cite{bmmr:sisc}, defined by
\[
F(X)=-A_0^{-1}(A_{-1}+A_1X^2),
\]
 with $X_0\in\{\widetilde G_1,\widetilde G_2\}$. We denote the two iterations obtained this way by  $FPI_1$, $FPI_2$, respectively. As for standard SDA, the fixed point iteration started with $X_0=0$ does not converge for Tests 1, 2, and 3 since, in this case, the matrices $X_k$ have a null limit part.

In Table \ref{tab:g}, for each test, we report the upper and the lower numerical bandwidth of the Toeplitz part $T(g)$ of the solution $G$, the numerical size and the rank of the compact correction and the length of the vector $v$ in the limit part. Here, for numerical bandwidth, numerical size and numerical length we mean the values obtained after truncating the corresponding entries to 0 if their value in modulus is less than the machine precision.  

\begin{table}[ht]\footnotesize
    \centering
    \begin{tabular}{c|cc|ccc|c}
    Test     & lb & ub & rc & cc & rk & lim \\ \hline
     1        & 738  & 53  &  1016  & 54   &  14 &  55\\
     2        & 2414  & 30  &  3729  & 32   &  11 &  31\\
     3        & 4096  & 1636  &  15320  & 2059   &  29 &  2009\\
    \end{tabular}
    \caption{\footnotesize Information about the solution $G$: Lower (lb) and upper (ub) bandwidth of the Toeplitz part, row-size (rc), column size (cc) and rank (rk) of the compact correction, length (lim) of the limit vector $v$.}
    \label{tab:g}
\end{table}

In Table \ref{tab:t1v}  we report the CPU time, the number of iterations and the residual error $r_k=\|A_{-1}+A_0X_k+A_1X_k^2\|_\infty$, where $X_k$ is the last approximation in the generated matrix sequence. The iteration is halted at step $k$ if the residual error $r_k$ satisfies one of the two following conditions: $r_k<10^{-14}$ or $r_k>r_{k-1}$.

\begin{table} 
\footnotesize
    \centering
    \begin{tabular}{c|cccc}
    Test & $SDA_1$ & $SDA_2$ & $FPI_1$ & $FPI_2$\\ \hline
    1 &  1.0 (7) & 0.9 (6) &  3.6 (176) & 2.8 (108)  \\ 
      & 6.1e-13  & 7.4e-14 &  6.5e-14  & 2.4e-14\\\hline
    2 & 4.2 (7)  & 2.6 (5)  & 63.1 (185)  & 7.7 (70) \\
      & 4.9e-13  & 8.9e-14  & 2.5e-14  & 2.5e-14  \\\hline
    3 & 24.3 (11)   & 20.9 (11)  & 1140.3 (3292)  & 1053.7 (2426) \\
      & 1.8e-11   &  6.5e-12  &  1.1e-12    & 2.1e-12
\end{tabular}
    \caption{
    \footnotesize CPU time, number of iterations, and residual error for SDA and Fixed Point algoritms for Tests 1,2,3. The subscript 1,2 denotes the initial approximation $\widetilde G=\frac12(I+\one e_1^T)$, and $\widetilde G=T(g)+(\one -T(g)\one )e_1^T$, respectively. 
    }
     \label{tab:t1v}
\end{table}   
    
It is important to point out that standard SDA, 
as well as the sequence generated by the functional iteration with $X_0$  in $\mathcal{QT}\setminus\mathcal{EQT}$, fail to converge. If $X_0$ is chosen with nonzero limit part, then fixed point iterations converge with linear speed, while our modification of SDA maintains a quadratic convergence speed.

From the timings and the number of iterations reported in Table \ref{tab:t1v}, it turns out that for Tests 1, 2, and 3, the reduction of the number of iterations of $SDA_2$, with respect to $FPI_2$, is by a factor of 18, 14, and
220.5,  respectively. While the speed-up in the CPU time is by a factor of about 3 for Tests 1 and 2, and about 50 for Test 3.
We may also observe a slightly better behavior of $SDA_2$ with respect to $SDA_1$.

The residual errors of the approximation to $G$ computed by the different algorithms are comparable. We may observe that, in the case of Test 3, the residual error is greater than the corresponding errors  for Test 1 and 2, by two orders of magnitude. This is independent of the algorithm used for the computation, and is due to the fact that the Test 3 is closer than Tests 1 and 2 to null recurrence.

\section{Conclusions}\label{sec:con}
The technique of defect correction has been applied to solving a quadratic matrix equation of the kind $A_{-1}+A_0X+A_1X^2=0$. The defect equation has been expressed in terms of invariant subspaces and the SDA has been applied for its solution. This approach has led to a quadratically convergent algorithm that allows the choice of the initial approximation. The case where $A_{-1}+A_0+A_1$ is stochastic has been further analysed, and an application to infinite quasi-Toeplitz matrices, encountered in the analysis of random walks in the quarter plane, has been considered. Numerical experiments, performed in the case of infinite quasi-Toeplitz matrices, have shown that the modified SDA proposed in this paper overcomes fixed point iteration both in speed and in CPU time by a quite large factor.


\end{document}